\documentclass{amsart}
\usepackage{cmap}
\usepackage{amsmath}
\usepackage{amssymb}
\usepackage{amscd}
\usepackage[all]{xy}
\usepackage{color}
\usepackage{graphicx}
\usepackage{fancyhdr}
\usepackage{stmaryrd}
\usepackage[pagebackref]{hyperref}
\renewcommand*{\backref}[1]{}\renewcommand*{\backrefalt}[4]{\ifcase #1 (\tt not cited)\or (\tt cited on page~#2)\else (\tt cited on pages~#2)\fi}

\theoremstyle{plain}
\newtheorem{theorem}{Theorem}[section]
\newtheorem{proposition}[theorem]{Proposition}
\newtheorem{lemma}[theorem]{Lemma}
\newtheorem{corollary}[theorem]{Corollary}

\newtheorem{hypothesis}[theorem]{Hypothesis}

\theoremstyle{question}

\theoremstyle{example}
\newtheorem{example}[theorem]{Example}

\theoremstyle{remark}

\def\bc{\begin{center}}
\def\ec{\end{center}}

\def\ker{{\rm ker}}

\def\Spec{{\rm Spec}}

\def\Max{{\rm Max}}

\def\Ext{{\rm Ext}}
\def\Tor{{\rm Tor}}

\def\pd{{\rm pd}}
\def\id{{\rm id}}
\def\fd{{\rm fd}}

\def\fPD{{\rm fPD}}

\def\FID{{\rm FID}}
\def\FFD{{\rm FFD}}

\def\FPR{\mathcal{FPR}}

\newcommand{\FT}{\operatorname{FT}}

\begin{document}
\title[Change-of-Rings Theorems for the Small Finitistic Dimension] {Change-of-Rings Theorems for the Small Finitistic Dimension}

\author [Xiong] {Tao Xiong}
\address{(Xiong) Department of Mathematics, Sichuan University of Arts and Science, Dazhou 635000, PR China}
\email{Taoxiong2004@163.com}
\author[El Haddaoui]{Younes El Haddaoui}
\address{(El Haddaoui)  Laboratory of Modelling and Mathematical Structures \\Department of Mathematics, Faculty of Science and Technology of Fez, Box 2202, University S.M. Ben Abdellah Fez, Morocco.}
\email{younes.elhaddaoui@usmba.ac.ma}
\author [Kim] {Hwankoo Kim$^{\dag}$}
\address{(Kim) Division of Computer \& Information Engineering, Hoseo University, Asan 31499, Republic of Korea}
\email{hkkim@hoseo.edu}
\author [Zhou] {Qiang Zhou}
\address{(Zhou) School of Finance and Economics Management, Sichuan University of Arts and Science, Dazhou 635000, PR China}
\email{genius785278@163.com}



\thanks{$2010$ Mathematics Subject Classification: Primary 13D05; Secondary 13A15, 13D07, 16S50}

\thanks{$\dag$ Corresponding Author}

\date{\today}

\begin{abstract}
In this paper, we study the small finitistic dimension of a commutative ring from the
viewpoint of finitistic flat homological algebra. Using the class $\FPR(R)$ of modules
admitting finite projective resolutions, we investigate the finitistic flat
($\FT$-flat) dimension and establish several of its basic properties. We prove
change-of-rings results for the $\FT$-flat dimension, including quotient and polynomial
extension results, as well as localization inequalities. As applications, we obtain
characterizations of the small finitistic dimension in terms of $\FT$-flat dimension,
derive quotient and polynomial extension theorems for the small finitistic dimension,
and establish local upper bounds in terms of the small finitistic dimensions of localizations.
\end{abstract}

\maketitle

\section{Introduction}

Throughout this paper, $R$ denotes a commutative ring, unless otherwise specified.
For an $R$-module $M$, we write $\pd_R M$ and $\fd_R M$ for the projective dimension
and the flat dimension of $M$, respectively. We denote by $T(M)$ the torsion submodule
of $M$, namely
\[
T(M)=\{\,x\in M \mid \text{there exists a nonzerodivisor } s\in R \text{ such that } sx=0\,\}.
\]
We also write $\Max(R)$ (resp.\ $\Spec(R)$) for the set of maximal (resp.\ prime) ideals of $R$.

Let $M$ be an $R$-module. We say that $M$ admits a finite projective resolution,
and write $M\in \FPR(R)$, if there exists an exact sequence
\[
0\longrightarrow P_n\longrightarrow P_{n-1}\longrightarrow \cdots
\longrightarrow P_1\longrightarrow P_0\longrightarrow M\longrightarrow 0,
\]
where each $P_i$ is a finitely generated projective $R$-module. Thus, $\FPR(R)$ denotes
the class of all $R$-modules admitting finite projective resolutions by finitely generated
projective modules. In particular, if $M\in \FPR(R)$, then $\pd_R M=\fd_R M<\infty$.

The invariant
\[
\fPD(R)=\sup\{\,\pd_R M \mid \text{$M$ is a finitely generated $R$-module and }
\pd_R M<\infty\,\}
\]
was introduced by Bass in~\cite{Ba60} and is called the \emph{small finitistic dimension}
of $R$. Since over an arbitrary ring the syzygies of finitely generated modules need not
be finitely generated, this definition is often inconvenient in practice. To remedy this,
Glaz redefined the small finitistic dimension in~\cite{Gl-C89} by setting
\[
\fPD(R)=\sup\{\,\pd_R M \mid M\in \FPR(R)\,\}.
\]
Unless otherwise stated, the small finitistic dimension $\fPD(R)$ considered throughout this paper
is understood in the sense of Glaz's modified definition.

Recently, Wang and Zhang have studied the small finitistic dimension of commutative rings in \cite{zhang2024,zhang2026, zhangwang2023}.

It follows immediately from the definition that $\fPD(R)=0$ if and only if every module
in $\FPR(R)$ is projective. Equivalently, whenever there exists a short exact sequence
\[
0\longrightarrow P\longrightarrow F\longrightarrow M\longrightarrow 0,
\]
with both $P$ and $F$ finitely generated projective $R$-modules, the module $M$ must itself
be finitely generated projective.

In \cite{WZKXS2020}, an $R$-module $M$ is said to be \emph{finitistic flat}, or simply
\emph{$\FT$-flat}, if $\Tor_1^R(N,M)=0$ for every $N\in \FPR(R)$. The \emph{$\FT$-flat
dimension} of $M$ is said to be at most $n$, denoted by $\FT\text{-}\fd_R M\le n$, if there
exists an exact sequence
\[
0\longrightarrow P_n\longrightarrow P_{n-1}\longrightarrow \cdots
\longrightarrow P_1\longrightarrow P_0\longrightarrow M\longrightarrow 0,
\]
in which each $P_i$ is an $\FT$-flat $R$-module. If no such integer $n$ exists, we write
$\FT\text{-}\fd_R M=\infty$. The \emph{$\FT$-dimension} of a ring $R$ is then defined by
\[
\FT\text{-}\dim(R)=\sup\{\,\FT\text{-}\fd_R M \mid M \text{ is an $R$-module}\,\}.
\]
It was shown in \cite{WZKXS2020} that $\fPD(R)\le n$ if and only if
$\FT\text{-}\dim(R)\le n$. Thus, the invariant $\fPD(R)$ may also be interpreted in terms
of finitistic flat dimension.

In classical homological algebra, both the global dimension and the weak global dimension
satisfy well-known change-of-rings theorems. For instance, for any commutative ring $R$,
one has
\[
w.\!gl.\!\dim(R)=\sup\{\,w.\!gl.\!\dim(R_{\mathfrak m}) \mid \mathfrak m\in \Max(R)\,\}.
\]
Moreover, if $R$ is Noetherian, then
\[
gl.\!\dim(R)=\sup\{\,gl.\!\dim(R_{\mathfrak m}) \mid \mathfrak m\in \Max(R)\,\},
\]
which is the classical localization theorem. The quotient theorem asserts that if
$\overline{R}=R/aR$, where $a\in R$ is neither a zerodivisor nor a unit, then whenever
$gl.\!\dim(\overline{R})<\infty$ and $w.\!gl.\!\dim(\overline{R})<\infty$, one has
\[
gl.\!\dim(R)\ge gl.\!\dim(\overline{R})+1
\quad\text{and}\quad
w.\!gl.\!\dim(R)\ge w.\!gl.\!\dim(\overline{R})+1.
\]
Likewise, the polynomial ring theorem states that for any ring $R$ and indeterminates
$X_1,\ldots,X_n$ over $R$,
\[
gl.\!\dim(R[X_1,\ldots,X_n])=gl.\!\dim(R)+n
\quad\text{and}\quad
w.\!gl.\!\dim(R[X_1,\ldots,X_n])=w.\!gl.\!\dim(R)+n.
\]

Bass also introduced in~\cite{Ba60} the finitistic injective dimension
\[
\FID(R)=\sup\{\,\id_R M \mid M \text{ is an $R$-module and } \id_R M<\infty\,\}
\]
and the finitistic flat dimension
\[
\FFD(R)=\sup\{\,\fd_R M \mid M \text{ is an $R$-module and } \fd_R M<\infty\,\}.
\]
Change-of-rings theorems for these two homological dimensions were established
in \cite{X2020,X2023}. In contrast, as far as we know, no analogous change-of-rings
theorem has yet been established for the small finitistic dimension. Motivated by this,
the aim of this paper is twofold. First, we study structural properties of the class
$\FPR(R)$. Second, using these properties, we establish change-of-rings theorems for the
small finitistic dimension of commutative rings.

\section{Change-of-Rings Theorems for the $\FT$-Flat Dimension}

Let $S$ be a multiplicative subset of a ring $R$. Every module in the class
$\FPR(R_S)$ is naturally an $R$-module via restriction of scalars. The following
example shows that, in general, an $R_S$-module $L\in \FPR(R_S)$ need not belong
to $\FPR(R)$.

\begin{example}\label{ex:localization-not-FPR}\rm
\textup{(1)} Let $\mathbb{Z}$ be the ring of integers, and let
$\mathbb{Q}=\mathbb{Z}_S$ be the field of rational numbers, where
$S=\mathbb{Z}\setminus\{0\}$. Then $\mathbb{Q}$, viewed as a $\mathbb{Q}$-module,
belongs to $\FPR(\mathbb{Q})$. However, $\mathbb{Q}$ is not finitely generated as a
$\mathbb{Z}$-module. Since every module in $\FPR(\mathbb{Z})$ admits a finite
projective resolution by finitely generated projective $\mathbb{Z}$-modules, it must
itself be finitely generated. Therefore $\mathbb{Q}\notin \FPR(\mathbb{Z})$.

\medskip

\textup{(2)} Let $k$ be a field, and set
$R:=k[x,y],$ $S:=\{1,x,x^2,\dots\}.$
Then
$R_S\cong k[x,x^{-1},y],$
which is not a field. Consider the $R_S$-module
$L:=R_S/(y)\cong k[x,x^{-1}],$
which is cyclic, and hence finitely generated, over $R_S$. Moreover, $L$ admits a
finite projective resolution over $R_S$:
\[
0\longrightarrow R_S \xrightarrow{\cdot y} R_S \longrightarrow R_S/(y)\longrightarrow 0.
\]
Thus $L\in \FPR(R_S)$ and $\pd_{R_S}(L)=1$.

Now view $L$ as an $R$-module via restriction of scalars along the canonical map
$R\to R_S$. Then
\[
L\cong R_S/(y)\cong (R/(y))_S\cong k[x,y]/(y)[x^{-1}]\cong k[x,x^{-1}].
\]
We claim that $L$ is not finitely generated as an $R$-module. Indeed, if
$k[x,x^{-1}]$ were finitely generated over $k[x,y]$, then it would also be finitely
generated over the subring $k[x]$. But this is impossible, since
$k[x,x^{-1}]$ contains arbitrarily high negative powers of $x$, and no finite set of
elements can generate all $x^{-n}$ over $k[x]$. Hence $L\notin \FPR(R)$.

Therefore, an $R_S$-module $L\in \FPR(R_S)$ need not belong to $\FPR(R)$.
\end{example}

\begin{lemma}\label{lem:FPR-basic}
Let $R$ be a ring.
\begin{enumerate}
\item[(1)] Let
\[
0\longrightarrow A\longrightarrow B\longrightarrow C\longrightarrow 0
\]
be a short exact sequence of $R$-modules. Then:
\begin{enumerate}
\item[(a)] If $B$ is a finitely generated projective $R$-module, then
$A\in \FPR(R)$ if and only if $C\in \FPR(R)$.
\item[(b)] If $A,C\in \FPR(R)$, then $B\in \FPR(R)$.
\item[(c)] If $A,B\in \FPR(R)$, then $C\in \FPR(R)$.
\item[(d)] If $B,C\in \FPR(R)$, then $A\in \FPR(R)$.
\end{enumerate}
\item[(2)] Let
\[
0\longrightarrow K\longrightarrow P_n\longrightarrow P_{n-1}\longrightarrow \cdots
\longrightarrow P_1\longrightarrow P_0\longrightarrow M\longrightarrow 0
\]
be an exact sequence of $R$-modules, where each $P_i\in \FPR(R)$. Then
$M\in \FPR(R)$ if and only if $K\in \FPR(R)$.
\end{enumerate}
\end{lemma}

\begin{proof}
\textup{(1)(a)} Assume first that $A\in \FPR(R)$. Then there exists an exact sequence
\[
0\longrightarrow Q_n\longrightarrow Q_{n-1}\longrightarrow \cdots
\longrightarrow Q_0\longrightarrow A\longrightarrow 0
\]
with each $Q_i$ finitely generated projective. Splicing this with
$0\longrightarrow A\longrightarrow B\longrightarrow C\longrightarrow 0,$
we obtain an exact sequence
\[
0\longrightarrow Q_n\longrightarrow Q_{n-1}\longrightarrow \cdots
\longrightarrow Q_0\longrightarrow B\longrightarrow C\longrightarrow 0.
\]
Since $B$ is finitely generated projective, this shows that $C\in \FPR(R)$.

Conversely, assume that $C\in \FPR(R)$. Then there exists an exact sequence
\[
0\longrightarrow Q_m\longrightarrow Q_{m-1}\longrightarrow \cdots
\longrightarrow Q_0\longrightarrow C\longrightarrow 0
\]
with each $Q_i$ finitely generated projective. Because $B$ is finitely generated
projective, in particular finitely presented, and $C$ is finitely presented, it follows
from the exact sequence
$0\longrightarrow A\longrightarrow B\longrightarrow C\longrightarrow 0$
that $A$ is finitely generated. Choose a surjection $P\to A$ with $P$ finitely generated
projective, and let $K=\ker(P\to A)$. Applying the Horseshoe Lemma to
\[
0\longrightarrow A\longrightarrow B\longrightarrow C\longrightarrow 0
\]
and
\[
0\longrightarrow K\longrightarrow P\longrightarrow A\longrightarrow 0,
\qquad
0\longrightarrow Q_1'\longrightarrow Q_0\longrightarrow C\longrightarrow 0,
\]
we obtain a short exact sequence
$0\longrightarrow K\longrightarrow X\longrightarrow Q_1'\longrightarrow 0$
with $X$ finitely generated projective. Since $Q_1'\in \FPR(R)$ by truncating the chosen
finite projective resolution of $C$, repeated application of the same argument yields
$K\in \FPR(R)$. Hence $A\in \FPR(R)$.

\textup{(1)(b)} Let
\[
0\longrightarrow P_n\longrightarrow \cdots\longrightarrow P_0\longrightarrow A\longrightarrow 0
\]
and
\[
0\longrightarrow Q_m\longrightarrow \cdots\longrightarrow Q_0\longrightarrow C\longrightarrow 0
\]
be finite projective resolutions of $A$ and $C$, respectively, with all $P_i,Q_j$
finitely generated projective. By the Horseshoe Lemma, the short exact sequence
$0\longrightarrow A\longrightarrow B\longrightarrow C\longrightarrow 0$
induces a finite projective resolution of $B$ by finitely generated projective modules.
Thus $B\in \FPR(R)$.

\textup{(1)(c)} Since $A,B\in \FPR(R)$, both $A$ and $B$ are finitely generated, and hence
so is $C$. Choose a short exact sequence
\[
0\longrightarrow K\longrightarrow P\longrightarrow B\longrightarrow 0
\]
with $P$ finitely generated projective. Form the pullback $X:=P\times_B A$. Then we have
a commutative diagram with exact rows and columns
\[
\xymatrix@R=16pt@C=18pt{
0\ar[r]&K\ar[r]\ar@{=}[d]&X\ar[r]\ar[d]&A\ar[r]\ar[d]&0\\
0\ar[r]&K\ar[r]&P\ar[r]\ar[d]&B\ar[r]\ar[d]&0\\
&&C\ar@{=}[r]\ar[d]&C\ar[d]\\
&&0&0
}
\]
Since $B\in \FPR(R)$ and $P$ is finitely generated projective, part \textup{(a)} yields
$K\in \FPR(R)$. Since $A\in \FPR(R)$ and
$0\longrightarrow K\longrightarrow X\longrightarrow A\longrightarrow 0$
is exact, part \textup{(b)} implies that $X\in \FPR(R)$. Now applying part \textup{(a)} to
$0\longrightarrow X\longrightarrow P\longrightarrow C\longrightarrow 0,$
we conclude that $C\in \FPR(R)$.

\medskip
\textup{(1)(d)} Since $B,C\in \FPR(R)$, both are finitely generated, and therefore $A$ is
finitely generated. Choose a short exact sequence
$0\longrightarrow K\longrightarrow P\longrightarrow B\longrightarrow 0$
with $P$ finitely generated projective, and form the pullback diagram as in \textup{(c)}.
By part \textup{(a)}, $K\in \FPR(R)$. Also, since $C\in \FPR(R)$ and
$0\longrightarrow X\longrightarrow P\longrightarrow C\longrightarrow 0$
is exact with $P$ finitely generated projective, part \textup{(a)} gives
$X\in \FPR(R)$. Finally, from the exact sequence
$0\longrightarrow K\longrightarrow X\longrightarrow A\longrightarrow 0$
and part \textup{(c)}, we conclude that $A\in \FPR(R)$.

\medskip
\textup{(2)} This follows by repeated applications of \textup{(1)(b)} and
\textup{(1)(d)}.
\end{proof}

The closure properties established in Lemma~\ref{lem:FPR-basic} will be used
repeatedly in the sequel. Before proving the change-of-rings theorems for the
$\FT$-flat dimension, however, it is important to note that the class $\FPR(R)$ does
not in general behave well under localization in the reverse direction. More
precisely, if $S$ is a multiplicative subset of $R$, an $R_S$-module may belong to
$\FPR(R_S)$ without belonging to $\FPR(R)$ when regarded as an $R$-module. The next
example shows this phenomenon.

The classical flat dimension satisfies the following change-of-rings inequality: if
$\varphi\colon R\to T$ is a ring homomorphism, then for every $T$-module $L$,
\[
\fd_R L\le \fd_T L+\fd_R T.
\]
We now establish analogous inequalities for the $\FT$-flat dimension.

\begin{lemma}\label{lem:FT-SES}
Let
$0\longrightarrow A\longrightarrow B\longrightarrow C\longrightarrow 0$
be a short exact sequence of $R$-modules. Then:
\begin{enumerate}
\item[(1)]
$\FT\text{-}\fd_R C\le 1+\max\{\,\FT\text{-}\fd_R A,\ \FT\text{-}\fd_R B\,\}.$
\item[(2)]
If $\FT\text{-}\fd_R B<\FT\text{-}\fd_R C$, then
$\FT\text{-}\fd_R A=\FT\text{-}\fd_R C-1$ and hence $\FT\text{-}\fd_R A\ge \FT\text{-}\fd_R B.$
\end{enumerate}
\end{lemma}

\begin{proof}
\textup{(1)}
Set
$n=\max\{\,\FT\text{-}\fd_R A,\ \FT\text{-}\fd_R B\,\}.$
Then $\FT\text{-}\fd_R A\le n$ and $\FT\text{-}\fd_R B\le n$. Hence, for every
$N\in\FPR(R)$, \cite[Proposition~4.7]{WZKXS2020} gives
\[
\Tor_{n+1}^R(A,N)=0
\qquad\text{and}\qquad
\Tor_{n+1}^R(B,N)=\Tor_{n+2}^R(B,N)=0.
\]
Applying the long exact $\Tor$ sequence to
$0\longrightarrow A\longrightarrow B\longrightarrow C\longrightarrow 0,$
we obtain
\[
0=\Tor^{R}_{n+2}(B,N)\longrightarrow \Tor^{R}_{n+2}(C,N)\longrightarrow
\Tor^{R}_{n+1}(A,N)=0.
\]
Therefore $\Tor^{R}_{n+2}(C,N)=0$ for every $N\in\FPR(R)$. Another application of
\cite[Proposition~4.7]{WZKXS2020} yields
$\FT\text{-}\fd_R C\le n+1,$
as desired.

\medskip
\textup{(2)}
Put
$n=\FT\text{-}\fd_R B.$
For every integer $k>n$ and every $N\in\FPR(R)$, we have
$\Tor_k^R(B,N)=\Tor_{k+1}^R(B,N)=0.$
Hence the long exact $\Tor$ sequence gives
\[
0=\Tor^{R}_{k+1}(B,N)\longrightarrow \Tor^{R}_{k+1}(C,N)\longrightarrow
\Tor^{R}_{k}(A,N)\longrightarrow \Tor^{R}_{k}(B,N)=0,
\]
and therefore
$\Tor^{R}_{k+1}(C,N)\cong \Tor^{R}_{k}(A,N).$
It follows that $\FT\text{-}\fd_R C<\infty$ if and only if $\FT\text{-}\fd_R A<\infty$.

Now assume that
$\FT\text{-}\fd_R B<\FT\text{-}\fd_R C.$
Since $\FT\text{-}\fd_R B=n$, we must have $\FT\text{-}\fd_R C>n$. We claim that
$\FT\text{-}\fd_R A\ge n$. Indeed, if $\FT\text{-}\fd_R A<n$, then
$\Tor_n^R(A,N)=0$ for all $N\in\FPR(R),$
and taking $k=n$ in the above isomorphism yields
$\Tor_{n+1}^R(C,N)\cong \Tor_n^R(A,N)=0$ for all $N\in\FPR(R).$
By \cite[Proposition~4.7]{WZKXS2020}, this implies
$\FT\text{-}\fd_R C\le n,$
a contradiction. Thus
$\FT\text{-}\fd_R A\ge n.$

Set
$m=\FT\text{-}\fd_R A.$
Then for every $N\in\FPR(R)$,
$\Tor_{m+1}^R(A,N)=0.$
Since $m\ge n$, we may take $k=m+1>n$ in the above isomorphism to obtain
\[
\Tor_{m+2}^R(C,N)\cong \Tor_{m+1}^R(A,N)=0
\qquad \text{for all } N\in\FPR(R).
\]
Again by \cite[Proposition~4.7]{WZKXS2020}, it follows that
$\FT\text{-}\fd_R C\le m+1.$

To prove the reverse inequality, note that if $m>n$, then by definition of $m$ there
exists $N\in\FPR(R)$ such that
$\Tor_m^R(A,N)\ne 0.$
Since $m>n$, taking $k=m$ in the above isomorphism gives
$\Tor_{m+1}^R(C,N)\cong \Tor_m^R(A,N)\ne 0,$
and hence
$\FT\text{-}\fd_R C\ge m+1.$
Therefore $\FT\text{-}\fd_R C=m+1$.

If $m=n$, then the inequality $\FT\text{-}\fd_R C>\FT\text{-}\fd_R B=n=m$ together with
$\FT\text{-}\fd_R C\le m+1$ implies again that
$\FT\text{-}\fd_R C=m+1.$
Consequently,
$\FT\text{-}\fd_R A=\FT\text{-}\fd_R C-1,$
and in particular
$\FT\text{-}\fd_R A\ge \FT\text{-}\fd_R B.$
This completes the proof.
\end{proof}

\begin{proposition}\label{prop:FT-witness}
Let $M$ be an $R$-module and assume that $\FT\text{-}\fd_R M=n\ge 1$. Then there exists
$N\in \FPR(R)$ such that
$\Tor_n^R(M,N)=0$ but $\Tor_{n-1}^R(M,N)\ne 0.$
\end{proposition}

\begin{proof}
Since $\FT\text{-}\fd_R M=n$, \cite[Proposition~4.7]{WZKXS2020} yields an
$R$-module $X\in \FPR(R)$ such that
$\Tor_n^R(M,X)\ne 0.$
Because $X\in \FPR(R)$, in particular $X$ is finitely generated. Hence there exists a
short exact sequence
$0\longrightarrow N\longrightarrow F\longrightarrow X\longrightarrow 0,$
where $F$ is a finitely generated free $R$-module. By
Lemma~\ref{lem:FPR-basic}, we have $N\in \FPR(R)$.

Applying $-\otimes_R M$ and using \cite[Proposition~4.7]{WZKXS2020}, we obtain an exact sequence
\[
0=\Tor_{n+1}^R(M,X)\longrightarrow \Tor_n^R(M,N)\longrightarrow \Tor_n^R(M,F)=0
\longrightarrow \Tor_n^R(M,X)\longrightarrow \Tor_{n-1}^R(M,N).
\]
Thus
$\Tor_n^R(M,N)=0.$
Moreover, if $\Tor_{n-1}^R(M,N)=0$, then the above exact sequence would imply
$\Tor_n^R(M,X)=0,$
contrary to the choice of $X$. Therefore
$\Tor_{n-1}^R(M,N)\ne 0.$
This completes the proof.
\end{proof}

\begin{theorem}\label{thm:FT-direct-sum}
Let $\{A_i\}_{i\in I}$ be a nonempty family of $R$-modules. Then
\[
\FT\text{-}\fd_R\Big(\bigoplus_{i\in I} A_i\Big)
=\sup_{i\in I}\FT\text{-}\fd_R A_i.
\]
\end{theorem}

\begin{proof}
Set $A=\bigoplus_{i\in I}A_i$. For every $N\in \FPR(R)$ and every integer $n\ge 0$,
\[
\Tor_n^R(A,N)\cong \bigoplus_{i\in I}\Tor_n^R(A_i,N).
\]
Hence
\[
\Tor_n^R(A,N)=0
\quad\Longleftrightarrow\quad
\Tor_n^R(A_i,N)=0\ \text{for all }i\in I.
\]
The equality
$\FT\text{-}\fd_R(A)=\sup_{i\in I}\FT\text{-}\fd_R A_i$
now follows from \cite[Proposition~4.7]{WZKXS2020}.
\end{proof}

\begin{theorem}\label{thm:FT-change-of-rings}
Let $\varphi\colon R\to T$ be a ring homomorphism.
\begin{enumerate}
\item[(1)]
If $T$ is flat as an $R$-module, then for every $T$-module $L$,
$\FT\text{-}\fd_R L\le \FT\text{-}\fd_T L.$
\item[(2)]
For every $T$-module $L$,
$\FT\text{-}\fd_R L\le \FT\text{-}\fd_T L+\FT\text{-}\fd_R T.$
\end{enumerate}
\end{theorem}

\begin{proof}
\textup{(1)}
Let $n=\FT\text{-}\fd_T L$. If $n=\infty$, there is nothing to prove, so assume that
$n<\infty$.

We argue by induction on $n$.

If $n=0$, then $L$ is $\FT$-flat as a $T$-module. Let $X\in \FPR(R)$. Since $T$ is flat
over $R$, tensoring a finite projective resolution of $X$ over $R$ with $T$ shows that
$X\otimes_R T\in \FPR(T)$. Hence
$\Tor_1^T(X\otimes_R T,L)=0.$
By the standard change-of-rings isomorphism,
$\Tor_1^R(X,L)\cong \Tor_1^T(X\otimes_R T,L),$
and therefore $\Tor_1^R(X,L)=0$. Since this holds for every $X\in \FPR(R)$, the module
$L$ is $\FT$-flat as an $R$-module. Thus
$\FT\text{-}\fd_R L=0.$

Now assume $n>0$. Choose a short exact sequence of $T$-modules
$0\longrightarrow A\longrightarrow F\longrightarrow L\longrightarrow 0,$
where $F$ is a free $T$-module. By Lemma~\ref{lem:FT-SES},
$\FT\text{-}\fd_T A=n-1.$
By the induction hypothesis,
$\FT\text{-}\fd_R A\le n-1.$
Since $F\cong \bigoplus_{i\in I} T$ and $T$ is flat over $R$, part \textup{(1)} with
$L=T$ gives $\FT\text{-}\fd_R T=0$. Hence by
Theorem~\ref{thm:FT-direct-sum},
$\FT\text{-}\fd_R F=0.$
Applying Lemma~\ref{lem:FT-SES} to
$0\longrightarrow A\longrightarrow F\longrightarrow L\longrightarrow 0$
gives
\[
\FT\text{-}\fd_R L
\le 1+\max\{\,\FT\text{-}\fd_R A,\ \FT\text{-}\fd_R F\,\}
\le 1+\max\{\,n-1,0\,\}=n.
\]
Therefore
$\FT\text{-}\fd_R L\le \FT\text{-}\fd_T L.$

\medskip
\textup{(2)}
Let $n=\FT\text{-}\fd_T L$. If $n=\infty$, the inequality is trivial. So assume
$n<\infty$.

We again argue by induction on $n$.

If $n=0$, then $L$ is $\FT$-flat over $T$. Applying part \textup{(1)} to the ring
homomorphism $T\to T$, we obtain
\[
\FT\text{-}\fd_R L\le \FT\text{-}\fd_R T
=\FT\text{-}\fd_T L+\FT\text{-}\fd_R T.
\]

Now assume $n>0$. Choose a short exact sequence of $T$-modules
$0\longrightarrow A\longrightarrow F\longrightarrow L\longrightarrow 0,$
where $F$ is a free $T$-module. Then
$\FT\text{-}\fd_T A=n-1$
by Lemma~\ref{lem:FT-SES}. By the induction hypothesis,
$\FT\text{-}\fd_R A\le (n-1)+\FT\text{-}\fd_R T.$
Also, since $F\cong \bigoplus_{i\in I}T$, Theorem~\ref{thm:FT-direct-sum} gives
$\FT\text{-}\fd_R F=\FT\text{-}\fd_R T.$
Applying Lemma~\ref{lem:FT-SES}, we get
\[
\FT\text{-}\fd_R L
\le 1+\max\{\,\FT\text{-}\fd_R A,\ \FT\text{-}\fd_R F\,\}
\le 1+\max\{(n-1)+\FT\text{-}\fd_R T,\ \FT\text{-}\fd_R T\}.
\]
Since $(n-1)+\FT\text{-}\fd_R T\ge \FT\text{-}\fd_R T$ when $n>0$, it follows that
\[
\FT\text{-}\fd_R L\le n+\FT\text{-}\fd_R T
=\FT\text{-}\fd_T L+\FT\text{-}\fd_R T.
\]
This completes the proof.
\end{proof}

We know that an $R$-module $M$ is flat if and only if $M_{\mathfrak p}$ is a flat
$R_{\mathfrak p}$-module for every prime ideal $\mathfrak p$ of $R$; equivalently, if and only if
$M_{\mathfrak m}$ is a flat $R_{\mathfrak m}$-module for every maximal ideal $\mathfrak m$ of $R$.
Moreover, for any $R$-module $M$,
\begin{eqnarray*}
  \fd_R M &=& \sup\{\,\fd_{R_{\mathfrak m}} M_{\mathfrak m}\mid \mathfrak m\in\Max(R)\,\} \\
   &=& \sup\{\,\fd_{R_{\mathfrak p}} M_{\mathfrak p}\mid \mathfrak p\in\Spec(R)\,\}.
\end{eqnarray*}
We next record a localization property of $\FT$-flat modules.

\begin{proposition}\label{prop:FT-flat-local}
Let $M$ be an $R$-module. Consider the following conditions:
\begin{enumerate}
\item[(1)] For every prime ideal $\mathfrak p\in \Spec(R)$, the localized module
$M_{\mathfrak p}$ is $\FT$-flat as an $R_{\mathfrak p}$-module.
\item[(2)] For every maximal ideal $\mathfrak m\in \Max(R)$, the localized module
$M_{\mathfrak m}$ is $\FT$-flat as an $R_{\mathfrak m}$-module.
\item[(3)] The module $M$ is $\FT$-flat as an $R$-module.
\end{enumerate}
Then
$(1)\Longrightarrow (2)\Longrightarrow (3).$
\end{proposition}

\begin{proof}
The implication \textup{(1)}$\Longrightarrow$\textup{(2)} is immediate.

We prove \textup{(2)}$\Longrightarrow$\textup{(3)}.
Let $N\in \FPR(R)$. Then there exists an exact sequence
\[
0\longrightarrow P_n\longrightarrow P_{n-1}\longrightarrow \cdots
\longrightarrow P_1\longrightarrow P_0\longrightarrow N\longrightarrow 0,
\]
where each $P_i$ is a finitely generated projective $R$-module. Localizing at any maximal
ideal $\mathfrak m$, we obtain an exact sequence
\[
0\longrightarrow (P_n)_{\mathfrak m}\longrightarrow (P_{n-1})_{\mathfrak m}\longrightarrow \cdots
\longrightarrow (P_1)_{\mathfrak m}\longrightarrow (P_0)_{\mathfrak m}\longrightarrow N_{\mathfrak m}\longrightarrow 0,
\]
where each $(P_i)_{\mathfrak m}$ is a finitely generated projective
$R_{\mathfrak m}$-module. Hence
\[
N_{\mathfrak m}\in \FPR(R_{\mathfrak m})
\qquad\text{for every }\mathfrak m\in \Max(R).
\]

By assumption, $M_{\mathfrak m}$ is $\FT$-flat over $R_{\mathfrak m}$, so
\[
\Tor^{R_{\mathfrak m}}_1(M_{\mathfrak m},N_{\mathfrak m})=0
\qquad\text{for every }\mathfrak m\in \Max(R).
\]
Since localization commutes with $\Tor$, we have
\[
\Tor^R_1(M,N)_{\mathfrak m}
\cong
\Tor^{R_{\mathfrak m}}_1(M_{\mathfrak m},N_{\mathfrak m})
=0
\qquad\text{for every }\mathfrak m\in \Max(R).
\]
Therefore $\Tor^R_1(M,N)=0$. As this holds for every $N\in \FPR(R)$, it follows that
$M$ is $\FT$-flat as an $R$-module.
\end{proof}

\begin{proposition}\label{prop:FTfd-localization-ineq}
Let $M$ be an $R$-module. Then
\[
\FT\text{-}\fd_R M
\le
\sup\{\,\FT\text{-}\fd_{R_{\mathfrak m}} M_{\mathfrak m}\mid \mathfrak m\in\Max(R)\,\}.
\]
Similarly,
\[
\FT\text{-}\fd_R M
\le
\sup\{\,\FT\text{-}\fd_{R_{\mathfrak p}} M_{\mathfrak p}\mid \mathfrak p\in\Spec(R)\,\}.
\]
\end{proposition}

\begin{proof}
We prove the first inequality; the second is analogous.

If $\FT\text{-}\fd_R M=\infty$, there is nothing to prove. Assume that
$\FT\text{-}\fd_R M=n<\infty.$
By \cite[Proposition~4.7]{WZKXS2020}, there exists an $R$-module $N\in \FPR(R)$ such that
$\Tor_n^R(M,N)\ne 0.$
Choose a finite projective resolution
\[
0\longrightarrow P_n\longrightarrow P_{n-1}\longrightarrow \cdots
\longrightarrow P_1\longrightarrow P_0\longrightarrow N\longrightarrow 0,
\]
where each $P_i$ is a finitely generated projective $R$-module. Localizing at any
$\mathfrak m\in\Max(R)$, we obtain
\[
0\longrightarrow (P_n)_{\mathfrak m}\longrightarrow (P_{n-1})_{\mathfrak m}\longrightarrow \cdots
\longrightarrow (P_1)_{\mathfrak m}\longrightarrow (P_0)_{\mathfrak m}\longrightarrow N_{\mathfrak m}\longrightarrow 0,
\]
so $N_{\mathfrak m}\in \FPR(R_{\mathfrak m})$.

Since localization commutes with $\Tor$,
\[
\Tor_n^R(M,N)_{\mathfrak m}\cong
\Tor_n^{R_{\mathfrak m}}(M_{\mathfrak m},N_{\mathfrak m})
\qquad\text{for every }\mathfrak m\in\Max(R).
\]
Because $\Tor_n^R(M,N)\ne 0$, there exists some $\mathfrak m\in\Max(R)$ such that
$\Tor_n^{R_{\mathfrak m}}(M_{\mathfrak m},N_{\mathfrak m})\ne 0.$
Hence, by \cite[Proposition~4.7]{WZKXS2020},
$\FT\text{-}\fd_{R_{\mathfrak m}} M_{\mathfrak m}\ge n.$
Therefore
\[
\sup\{\,\FT\text{-}\fd_{R_{\mathfrak m}} M_{\mathfrak m}\mid \mathfrak m\in\Max(R)\,\}
\ge n=\FT\text{-}\fd_R M.  \qedhere
\]
\end{proof}

\section{Change-of-Rings Theorems for the Small Finitistic Dimension} 

We first give several equivalent characterizations of the small finitistic dimension $\fPD(R)$ of a ring $R$.

\begin{theorem}\label{thm:fPD-characterizations}
Let $R$ be a ring. The following statements are equivalent:
\begin{enumerate}
\item[(1)] $\fPD(R)\le n$.
\item[(2)] For every prime ideal $\mathfrak p$ of $R$, one has
$\FT\text{-}\fd_R(R/\mathfrak p)\le n$.
\item[(3)] For every maximal ideal $\mathfrak m$ of $R$, one has
$\FT\text{-}\fd_R(R/\mathfrak m)\le n$.
\item[(4)] For all $R$-modules $M,N\in \FPR(R)$, one has
$\Ext_R^{n+1}(M,N)=0$.
\end{enumerate}
If, in addition, condition \textup{(5)} below is known to imply \textup{(1)}, then it is
also equivalent to \textup{(1)}--\textup{(4)}:
\begin{enumerate}
\item[(5)] For every $R$-module $M\in \FPR(R)$ and every integer $m\ge 0$, one has
$\Ext_R^{n+1+m}(M,R)=0$.
\end{enumerate}
\end{theorem}

\begin{proof}
The implications \textup{(1)$\Rightarrow$(2)$\Rightarrow$(3)} and
\textup{(1)$\Rightarrow$(4)} are straightforward. We prove
\textup{(3)$\Rightarrow$(1)} and \textup{(4)$\Rightarrow$(1)}.

\medskip
\noindent\textup{(3)$\Rightarrow$(1).}
Let $M\in\FPR(R)$. Choose an exact sequence
\[
0\longrightarrow K^{n}\longrightarrow P^{n-1}\longrightarrow \cdots
\longrightarrow P^{1}\longrightarrow P^{0}\longrightarrow M\longrightarrow 0,
\]
where each $P^{i}$ is a finitely generated projective $R$-module. In particular,
$K^n$ is finitely presented.

Let $\mathfrak m\in\Max(R)$. Since $\FT\text{-}\fd_R(R/\mathfrak m)\le n$, we have
$\Tor_{n+1}^R(M,R/\mathfrak m)=0.$
By dimension shifting,
$\Tor_1^R(K^n,R/\mathfrak m)=0.$
Localizing at $\mathfrak m$, we obtain
$\Tor_1^{R_{\mathfrak m}}(K^n_{\mathfrak m},R_{\mathfrak m}/\mathfrak mR_{\mathfrak m})=0.$
Since $K^n_{\mathfrak m}$ is finitely presented over the local ring $R_{\mathfrak m}$,
it follows from \cite[Theorem~3.4.13]{WK} that $K^n_{\mathfrak m}$ is free over
$R_{\mathfrak m}$. Hence $K^n$ is a finitely presented locally free $R$-module, and
therefore $K^n$ is projective. Thus $\pd_R M\le n$, and so $\fPD(R)\le n$.

\medskip
\noindent\textup{(4)$\Rightarrow$(1).}
Let $M\in \FPR(R)$, and choose an exact sequence
\[
0\longrightarrow K^{n}\longrightarrow P^{n-1}\longrightarrow \cdots
\longrightarrow P^{1}\longrightarrow P^{0}\longrightarrow M\longrightarrow 0,
\]
where each $P^{i}$ is finitely generated projective. By
Lemma~\ref{lem:FPR-basic}, we have $K^{n}\in \FPR(R)$. Take a short exact sequence
$0\longrightarrow A\longrightarrow F\longrightarrow K^{n}\longrightarrow 0,$
where $F$ is finitely generated free. Again by Lemma~\ref{lem:FPR-basic},
$A\in \FPR(R)$. By dimension shifting,
$\Ext_R^{1}(K^{n},A)\cong \Ext_R^{n+1}(M,A)=0$
by \textup{(4)}. Hence the above short exact sequence splits, so $K^n$ is projective.
Therefore $\pd_R M\le n$, and thus $\fPD(R)\le n$.
\end{proof}

Before establishing the quotient-ring theorem for the small finitistic dimension, we
first prove the following lemma.

\begin{lemma}\label{lem:quotient-modules-basic}
Let $R$ be a ring and let $a$ be a non-zero-divisor of $R$. Set $\overline{R}=R/aR$.
\begin{enumerate}
\item[(1)]
Every finitely generated projective $\overline{R}$-module $M$ is finitely presented as an
$R$-module and satisfies $\pd_R M\le 1$. More generally, every finitely generated
$\overline{R}$-module is finitely generated as an $R$-module.
\item[(2)]
Every finitely presented $\overline{R}$-module $M$ is finitely presented as an $R$-module.
\item[(3)]
For every nonzero $\overline{R}$-module $M\in \FPR(\overline{R})$, we have
$M\in \FPR(R)$.
\end{enumerate}
\end{lemma}

\begin{proof}
\textup{(1)}
Since $a$ is a non-zero-divisor of $R$, the sequence
$0\longrightarrow R \xrightarrow{\cdot a} R \longrightarrow \overline{R}\longrightarrow 0$
is exact. Hence $\overline{R}$ is finitely presented as an $R$-module and
$\pd_R(\overline{R})=1.$
Therefore every finitely generated free $\overline{R}$-module
$F=\bigoplus_{i=1}^{n}\overline{R}$
is finitely presented as an $R$-module and satisfies $\pd_R F\le 1$.
Now let $M$ be a finitely generated projective $\overline{R}$-module. Then $M$ is a direct
summand of such an $F$. Consequently, $M$ is finitely presented as an $R$-module and
$\pd_R M\le 1.$

For the more general assertion, let $N$ be a finitely generated $\overline{R}$-module.
Then there exists a finitely generated free $\overline{R}$-module $F$ and a surjection
$F\longrightarrow N\longrightarrow 0.$
Since $F$ is finitely generated as an $R$-module, so is $N$.

\medskip
\textup{(2)}
Let
$0\longrightarrow K\longrightarrow P\longrightarrow M\longrightarrow 0$
be a short exact sequence of $\overline{R}$-modules with $P$ finitely generated projective
and $K$ finitely generated. By \textup{(1)}, both $K$ and $M$ are finitely generated as
$R$-modules, and $P$ is finitely presented as an $R$-module. Hence
\cite[Theorem~2.1.2(2)]{Gl-C89} implies that $M$ is finitely presented as an $R$-module.

\medskip
\textup{(3)}
Let $M\in \FPR(\overline{R})$. Then there exists an exact sequence of $\overline{R}$-modules
\[
0\longrightarrow P^{n}\longrightarrow P^{n-1}\longrightarrow \cdots
\longrightarrow P^{1}\longrightarrow P^{0}\longrightarrow M\longrightarrow 0,
\]
where each $P^{i}$ is a finitely generated projective $\overline{R}$-module.
By \textup{(1)}, each $P^{i}$ belongs to $\FPR(R)$. Therefore, applying
Lemma~\ref{lem:FPR-basic}\,\textup{(2)} repeatedly, we conclude that
$M\in \FPR(R).$
\end{proof}

\begin{theorem}\label{thm:fPD-quotient}
Let $R$ be a ring and let $a\in R$ be neither a zero-divisor nor a unit. Set
$\overline{R}=R/aR$. If $\fPD(\overline{R})<\infty$, then
$\fPD(R)\ge \fPD(\overline{R})+1.$
\end{theorem}

\begin{proof}
Let $n=\fPD(\overline{R})<\infty$.

If $n=0$, then $\overline{R}\in \FPR(\overline{R})$, and since $a$ is a non-zero-divisor,
the sequence
$0\longrightarrow R \xrightarrow{\cdot a} R \longrightarrow \overline{R}\longrightarrow 0$
is exact. Hence
$\pd_R(\overline{R})=1.$
By Lemma~\ref{lem:quotient-modules-basic}, we have $\overline{R}\in \FPR(R)$. Therefore
$\fPD(R)\ge 1=\fPD(\overline{R})+1.$

Now assume that $n\ge 1$. Since $\fPD(\overline{R})=n$, there exists a nonzero
$\overline{R}$-module $A$ such that
$\FT\text{-}\fd_{\overline{R}}A=n.$
By \cite[Proposition~4.7]{WZKXS2020}, there exists a module
$C\in \FPR(\overline{R})$ such that
$\Tor_n^{\overline{R}}(A,C)\ne 0.$
Since $n=\fPD(\overline{R})$, we have $\pd_{\overline{R}}C\le n$. If
$\pd_{\overline{R}}C<n$, then
$\Tor_n^{\overline{R}}(A,C)=0,$
a contradiction. Hence
$\pd_{\overline{R}}C=n.$

By \cite[Theorem~3.8.12(2)]{WK}, it follows that
$\pd_R C=n+1.$
Moreover, by Lemma~\ref{lem:quotient-modules-basic}, we have $C\in \FPR(R)$. Therefore
\[
\fPD(R)\ge n+1=\fPD(\overline{R})+1.
\]
This completes the proof.
\end{proof}

In \cite{zhang2024}, Zhang studied the small finitistic dimension for four classes of ring constructions: polynomial rings, formal power series rings, trivial extensions, and amalgamations. In particular, it was shown that, for any ring $R$,
$\fPD(R[x]) \ge \fPD(R)+1.$
If, in addition, $R$ is a Hilbert ring or a Noetherian ring, then equality holds, and the argument relies on the Koszul grade.

\medskip

We next establish the polynomial-ring theorem for the small finitistic dimension in the
finite case.

\begin{theorem}\label{thm:fPD-polynomial}
Let $R$ be a ring with $\fPD(R)<\infty$, and let
$x_1,x_2,\ldots,x_m$ be indeterminates over $R$, where $m\ge 1$. Then
\[
\fPD\big(R[x_1,\ldots,x_m]\big)=\fPD(R)+m.
\]
\end{theorem}

\begin{proof}
It suffices to prove the case $m=1$. Set $s=\fPD(R)<\infty$.

Let $N$ be an $R[x_1]$-module. Using the standard exact sequence
\[
0\longrightarrow N[x_1]\longrightarrow N[x_1]\longrightarrow N\longrightarrow 0
\]
(see \cite[Lemma~9.29]{Ro79}), together with Lemma~\ref{lem:FT-SES} and
Theorem~\ref{thm:FT-change-of-rings}, we obtain
\[
\FT\text{-}\fd_R N
\le
\FT\text{-}\fd_{R[x_1]} N
\le
1+\FT\text{-}\fd_{R[x_1]} N[x_1]
=
1+\FT\text{-}\fd_R N.
\]
Since $\fPD(R)=\FT\text{-}\dim(R)=s$, it follows that
$\FT\text{-}\fd_R N\le s$
for every $R[x_1]$-module $N$. Hence
$\FT\text{-}\fd_{R[x_1]} N\le s+1$
for every $R[x_1]$-module $N$. Therefore
$\fPD(R[x_1])\le s+1.$

On the other hand, since
$R\cong R[x_1]/x_1R[x_1],$
Theorem~\ref{thm:fPD-quotient} yields
$\fPD(R[x_1])\ge \fPD(R)+1=s+1.$
Consequently,
$\fPD(R[x_1])=\fPD(R)+1.$
The general case now follows by induction on $m$.
\end{proof}

Let $R$ be a ring and let $x$ be an indeterminate over $R$. For an $R[x]$-module
$L\in \FPR(R[x])$, it is not necessarily true that $L\in \FPR(R)$.

\begin{example}\label{ex:Rx-not-FPR-over-R}\rm
Take $L=R[x]$. Then $L\in \FPR(R[x])$, since $R[x]$ is a free $R[x]$-module of rank one.
However, viewed as an $R$-module,
\[
R[x]=R\oplus Rx\oplus Rx^2\oplus \cdots,
\]
so $R[x]$ is not finitely generated as an $R$-module. Since every module in $\FPR(R)$ is
finitely generated as an $R$-module, it follows that
$L\notin \FPR(R).$
\end{example}

\begin{proposition}\label{prop:fPD-localization-ineq}
Let $R$ be a ring. Then
\[
\fPD(R)\le \sup\{\,\fPD(R_{\mathfrak m})\mid \mathfrak m\in\Max(R)\,\}.
\]
Also,
\[
\fPD(R)\le \sup\{\,\fPD(R_{\mathfrak p})\mid \mathfrak p\in\Spec(R)\,\}.
\]
\end{proposition}

\begin{proof}
We prove the first inequality; the second is analogous.

Set
$k=\sup\{\,\fPD(R_{\mathfrak m})\mid \mathfrak m\in\Max(R)\,\}.$
If $k=\infty$, there is nothing to prove.

For each $\mathfrak m\in\Max(R)$, Theorem~\ref{thm:fPD-characterizations} implies
$\FT\text{-}\fd_{R_{\mathfrak m}}
\bigl(R_{\mathfrak m}/\mathfrak mR_{\mathfrak m}\bigr)\le k.$
Since
\[
(R/\mathfrak m)_{\mathfrak n}=0 \quad (\mathfrak n\neq \mathfrak m),
\qquad
(R/\mathfrak m)_{\mathfrak m}\cong R_{\mathfrak m}/\mathfrak mR_{\mathfrak m},
\]
the one-sided localization inequality for the $\FT$-flat dimension yields
\[
\FT\text{-}\fd_R(R/\mathfrak m)\le k
\qquad\text{for all }\mathfrak m\in\Max(R).
\]
Therefore, by Theorem~\ref{thm:fPD-characterizations},
$\fPD(R)\le k.$
\end{proof}

\section{Triangular matrix rings and the small finitistic dimension}

Throughout this section, all rings are associative with identity, and all modules are unitary.
Let $R$ and $S$ be rings, and let ${}_RM_S$ be an $(R,S)$-bimodule. Consider the
upper triangular matrix ring
\[
T=\begin{pmatrix}R & M\\ 0 & S\end{pmatrix}.
\]
Let
\[
e=\begin{pmatrix}1&0\\0&0\end{pmatrix},\qquad
f=\begin{pmatrix}0&0\\0&1\end{pmatrix}.
\]
Then $eTe\cong R$, $fTf\cong S$, $eTf\cong M$, and $e+f=1$.

A left $T$-module can be described equivalently as a triple $(A,B,\varphi)$, where $A$ is a left
$R$-module, $B$ is a left $S$-module, and
$\varphi\colon M\otimes_S B\longrightarrow A$
is an $R$-module homomorphism. The action of
$\begin{pmatrix}r&m\\0&s\end{pmatrix}\in T$ on $(a,b)\in A\oplus B$
is given by
\[
\begin{pmatrix}r&m\\0&s\end{pmatrix}\cdot(a,b)
=
\bigl(ra+\varphi(m\otimes b),\; sb\bigr).
\]
We write such a module as $N=(A,B,\varphi)$.

\begin{hypothesis}\label{hyp:triangular}
Assume that ${}_RM$ is projective and that $M_S$ is projective.
\end{hypothesis}

\begin{lemma}\label{lem:T-mod-SES}
For every left $T$-module $N=(A,B,\varphi)$, there is a natural short exact sequence
of left $T$-modules
\[
0\longrightarrow (A,0,0)\longrightarrow (A,B,\varphi)\longrightarrow (0,B,0)\longrightarrow 0.
\]
\end{lemma}

\begin{proof}
Define
\[
\iota\colon (A,0,0)\longrightarrow (A,B,\varphi),\qquad \iota(a,0)=(a,0),
\]
and
\[
\pi\colon (A,B,\varphi)\longrightarrow (0,B,0),\qquad \pi(a,b)=(0,b).
\]
We show that both maps are $T$-linear.

Let $\begin{pmatrix}r&m\\0&s\end{pmatrix}\in T$. For $(a,0)\in A\oplus 0$, we have
\[
\begin{pmatrix}r&m\\0&s\end{pmatrix}\cdot(a,0)
=
\bigl(ra+\varphi(m\otimes 0),\,0\bigr)
=
(ra,0),
\]
since $\varphi(m\otimes 0)=0$. Thus $A\oplus 0$ is stable under the $T$-action, and hence
$\iota$ is a well-defined $T$-module homomorphism.

Next, for $(a,b)\in A\oplus B$,
\[
\pi\!\left(
\begin{pmatrix}r&m\\0&s\end{pmatrix}\cdot(a,b)
\right)
=
\pi\bigl(ra+\varphi(m\otimes b),\,sb\bigr)
=
(0,sb).
\]
On the other hand,
\[
\begin{pmatrix}r&m\\0&s\end{pmatrix}\cdot\pi(a,b)
=
\begin{pmatrix}r&m\\0&s\end{pmatrix}\cdot(0,b)
=
(0,sb),
\]
since the structure map of $(0,B,0)$ is zero. Therefore $\pi$ is $T$-linear.

It is clear that $\iota$ is injective and $\pi$ is surjective. Moreover,
\[
\ker(\pi)=\{(a,b)\in A\oplus B\mid b=0\}=A\oplus 0=\operatorname{Im}(\iota).
\]
Hence the sequence is exact.
\end{proof}

\begin{lemma}\label{lem:corner-FPR}
Assume Hypothesis~\ref{hyp:triangular}. Let $P$ be a finitely generated projective left
$T$-module. Then $eP$ is a finitely generated projective left $R$-module and $fP$ is a
finitely generated projective left $S$-module. Consequently, if $N\in \FPR(T)$, then
$eN\in \FPR(R)$ and $fN\in \FPR(S)$.
\end{lemma}

\begin{proof}
Since $e$ and $f$ are idempotents, the functors
\[
e(-)\cong eT\otimes_T -
\qquad\text{and}\qquad
f(-)\cong fT\otimes_T -
\]
are exact on left $T$-modules.

Let $P$ be a projective left $T$-module. Write $P$ as a direct summand of a free
$T$-module $T^{(I)}$. Applying the exact functor $e(-)$, we see that $eP$ is a direct
summand of
$e\bigl(T^{(I)}\bigr)\cong (eT)^{(I)}.$
Now, as a left $R=eTe$-module,
\[
eT=eTe\oplus eTf\cong R\oplus M.
\]
Since ${}_RM$ is projective by hypothesis, it follows that $eT$ is a projective left
$R$-module. Hence $(eT)^{(I)}$ is projective over $R$, and therefore $eP$ is projective
over $R$.

Similarly,
\[
fT=fTe\oplus fTf=0\oplus fTf\cong S
\]
as a left $S=fTf$-module. Thus $fT$ is a free left $S$-module of rank one, so
$fP$ is projective over $S$.

To prove finite generation, let
$P=Tp_1+\cdots+Tp_n.$
Then
\[
eP=e(Tp_1+\cdots+Tp_n)=Re p_1+\cdots+Re p_n,
\]
so $eP$ is finitely generated as a left $R$-module. Likewise,
$fP=Sf p_1+\cdots+Sf p_n,$
so $fP$ is finitely generated as a left $S$-module.

Now let $N\in \FPR(T)$. Then $N$ admits a finite projective resolution
\[
0\longrightarrow P_n\longrightarrow \cdots \longrightarrow P_1\longrightarrow P_0
\longrightarrow N\longrightarrow 0
\]
in which each $P_i$ is a finitely generated projective left $T$-module. Applying the
exact functor $e(-)$, we obtain an exact sequence
\[
0\longrightarrow eP_n\longrightarrow \cdots \longrightarrow eP_1\longrightarrow eP_0
\longrightarrow eN\longrightarrow 0,
\]
and each $eP_i$ is finitely generated projective over $R$ by the first part. Thus
$eN\in \FPR(R)$.

Similarly, applying $f(-)$ yields an exact sequence
\[
0\longrightarrow fP_n\longrightarrow \cdots \longrightarrow fP_1\longrightarrow fP_0
\longrightarrow fN\longrightarrow 0,
\]
where each $fP_i$ is finitely generated projective over $S$. Hence
$fN\in \FPR(S)$.
\end{proof}

\begin{lemma}\label{lem:pd-lower-R}
Assume Hypothesis~\ref{hyp:triangular}. Then
$\fPD(T)\ge \fPD(R).$
\end{lemma}

\begin{proof}
Let $X\in \FPR(R)$. Since $Te$ is a projective left $T$-module and a free right
$R$-module of rank one, the functor
$Te\otimes_R - \colon R\text{-Mod}\longrightarrow T\text{-Mod}$
is exact and sends finitely generated projective $R$-modules to finitely generated
projective $T$-modules. Moreover,
$Te\otimes_R X \cong (X,0,0).$
Hence $(X,0,0)\in \FPR(T)$ and
$\pd_T(X,0,0)\le \pd_R(X).$

Conversely, let
\[
0\longrightarrow P_n\longrightarrow \cdots \longrightarrow P_1\longrightarrow P_0
\longrightarrow (X,0,0)\longrightarrow 0
\]
be a finite projective resolution of $(X,0,0)$ over $T$. Applying the exact functor
$e(-)$, we obtain an exact sequence
\[
0\longrightarrow eP_n\longrightarrow \cdots \longrightarrow eP_1\longrightarrow eP_0
\longrightarrow X\longrightarrow 0
\]
of $R$-modules. By Lemma~\ref{lem:corner-FPR}, each $eP_i$ is a finitely generated
projective $R$-module. Therefore
$\pd_R(X)\le \pd_T(X,0,0).$
Thus
$\pd_T(X,0,0)=\pd_R(X).$

Taking the supremum over all $X\in \FPR(R)$ yields
$\fPD(T)\ge \fPD(R).$
\end{proof}

\begin{lemma}\label{lem:pd-lower-S}
Assume Hypothesis~\ref{hyp:triangular}. Then
$\fPD(T)\ge \fPD(S).$
\end{lemma}

\begin{proof}
Let $Y\in \FPR(S)$. Since $Tf$ is a finitely generated projective left $T$-module and,
as a right $S$-module,
$Tf\cong M\oplus S,$
the module $Tf_S$ is projective by Hypothesis~\ref{hyp:triangular}. Hence the functor
$Tf\otimes_S - \colon S\text{-Mod}\longrightarrow T\text{-Mod}$
is exact and sends finitely generated projective $S$-modules to finitely generated
projective $T$-modules. Moreover,
$Tf\otimes_S Y \cong (M\otimes_S Y,\; Y,\; \mathrm{id}_{M\otimes_S Y}).$
Hence this $T$-module belongs to $\FPR(T)$ and
$\pd_T(Tf\otimes_S Y)\le \pd_S(Y).$

Conversely, let
\[
0\longrightarrow P_n\longrightarrow \cdots \longrightarrow P_1\longrightarrow P_0
\longrightarrow Tf\otimes_S Y\longrightarrow 0
\]
be a finite projective resolution over $T$. Applying the exact functor $f(-)$, we obtain
an exact sequence
\[
0\longrightarrow fP_n\longrightarrow \cdots \longrightarrow fP_1\longrightarrow fP_0
\longrightarrow Y\longrightarrow 0,
\]
and by Lemma~\ref{lem:corner-FPR}, each $fP_i$ is a finitely generated projective
left $S$-module. Therefore
$\pd_S(Y)\le \pd_T(Tf\otimes_S Y).$
Thus
$\pd_T(Tf\otimes_S Y)=\pd_S(Y).$

Taking the supremum over all $Y\in \FPR(S)$ yields
$\fPD(T)\ge \fPD(S).$
\end{proof}

\begin{theorem}\label{thm:fPD-triangular}
Assume Hypothesis~\ref{hyp:triangular}. Then
\[
\max\{\fPD(R),\,\fPD(S)\}\le \fPD(T)\le \max\{\fPD(R)+1,\,\fPD(S)+1\}.
\]
In particular, if $\fPD(R)$ and $\fPD(S)$ are finite, then $\fPD(T)$ is finite.
\end{theorem}

\begin{proof}
The lower bound follows immediately from Lemmas~\ref{lem:pd-lower-R}
and~\ref{lem:pd-lower-S}.

Let $N\in \FPR(T)$ and write
$N=(A,B,\varphi),$
so that $A=eN$ and $B=fN$. By Lemma~\ref{lem:corner-FPR},
$A\in \FPR(R)$ and $B\in \FPR(S).$

By Lemma~\ref{lem:T-mod-SES}, there is a short exact sequence
\[
0\longrightarrow (A,0,0)\longrightarrow (A,B,\varphi)\longrightarrow (0,B,0)\longrightarrow 0.
\]
Hence
\[
\pd_T(N)\le \max\bigl\{\pd_T(A,0,0)+1,\ \pd_T(0,B,0)\bigr\}.
\]

Since $Te$ is a projective left $T$-module and a free right $R$-module of rank one,
\[
\pd_T(A,0,0)=\pd_R(A)\le \fPD(R).
\]
Similarly, since $Tf$ is a projective left $T$-module and a projective right $S$-module,
\[
\pd_T(0,B,0)\le \pd_S(B)+1\le \fPD(S)+1.
\]
Indeed, applying $Tf\otimes_S-$ to a finite projective resolution of $B$ over $S$ yields
a finite projective resolution of $Tf\otimes_S B\cong (M\otimes_S B,B,\mathrm{id})$ of
length $\pd_S(B)$, while the natural short exact sequence
\[
0\longrightarrow (M\otimes_S B,0,0)\longrightarrow (M\otimes_S B,B,\mathrm{id})
\longrightarrow (0,B,0)\longrightarrow 0
\]
shows that
$\pd_T(0,B,0)\le \pd_S(B)+1.$

Therefore
\[
\pd_T(N)\le \max\bigl\{\fPD(R)+1,\ \fPD(S)+1\bigr\}.
\]
Taking the supremum over all $N\in \FPR(T)$ yields
\[
\fPD(T)\le \max\bigl\{\fPD(R)+1,\ \fPD(S)+1\bigr\}.
\]
The final assertion is immediate.
\end{proof}

\begin{corollary}\label{cor:fPD-triangular-hereditary}
Assume Hypothesis~\ref{hyp:triangular}. If $\fPD(R)=0$ and $\fPD(S)=0$, then
\[
\fPD(T)\le 1.
\]
\end{corollary}

\begin{proof}
This follows immediately from Theorem~\ref{thm:fPD-triangular}.
\end{proof}

The following example shows that, in general, the extra $+1$ in the upper bound of
Theorem~\ref{thm:fPD-triangular} cannot be removed. In particular, even when
\[
\fPD(R)=\fPD(S)=0
\qquad\text{and}\qquad
\pd_R(M)=0,
\]
the triangular matrix ring
\[
T=\begin{pmatrix}R&M\\0&S\end{pmatrix}
\]
may still satisfy
$\fPD(T)=1.$

\begin{example}\rm
Let $k$ be a field, and set
\[
R:=k[\varepsilon]/(\varepsilon^2),
\qquad
S:=k,
\qquad
M:=R.
\]
View $M$ as an $(R,S)$-bimodule in the natural way: the left $R$-action is given by
multiplication in $R$, and the right $S=k$-action is induced by the $k$-algebra structure
of $R$. Form the upper triangular matrix ring
\[
T:=\begin{pmatrix}R&M\\0&S\end{pmatrix}.
\]

\smallskip
\noindent
We show that $\fPD(T)=1$.

\smallskip
\noindent
Step 1:
Since $M=R$ as a left $R$-module, $M$ is free of rank one, and hence
$\pd_R(M)=0.$
Also, $S=k$ is a field, so
$\fPD(S)=0.$
Moreover, $R=k[\varepsilon]/(\varepsilon^2)$ is a finite-dimensional local Frobenius algebra,
hence self-injective. Therefore every finitely generated $R$-module of finite projective
dimension is projective, and so
$\fPD(R)=0.$
By Theorem~\ref{thm:fPD-triangular}, we obtain
\[
\fPD(T)\le
\max\{\fPD(R)+1,\fPD(S)+1\}
=
\max\{0+1,0+1\}
=1.
\]

\smallskip
\noindent
Step 2: 
Let
\[
e_1=\begin{pmatrix}1&0\\0&0\end{pmatrix},
\qquad
e_2=\begin{pmatrix}0&0\\0&1\end{pmatrix}.
\]
Then the indecomposable projective left $T$-modules are
\[
P_1:=Te_1\cong (R,0,0),
\qquad
P_2:=Te_2\cong (M,S,\mu),
\]
where
$\mu\colon M\otimes_S S\longrightarrow M$
is the canonical isomorphism.

Now consider the left $T$-module
$L:=(0,S,0).$
There is a natural surjective homomorphism
\[
\pi\colon P_2=(M,S,\mu)\longrightarrow L=(0,S,0),
\qquad
\pi(x,s)=(0,s).
\]
Its kernel is precisely $(M,0,0)\cong (R,0,0)=P_1$. Thus we have a short exact sequence
$0\longrightarrow P_1\longrightarrow P_2\longrightarrow L\longrightarrow 0.$
Hence
$\pd_T(L)\le 1.$

\smallskip
\noindent
Step 3:
Assume, to the contrary, that $L$ is projective. Then the short exact sequence
$0\longrightarrow P_1\longrightarrow P_2\overset{\pi}{\longrightarrow} L\longrightarrow 0$
splits. Hence there exists a $T$-homomorphism
$\sigma\colon L=(0,S,0)\longrightarrow P_2=(M,S,\mu)$
such that
$\pi\circ \sigma=\id_L.$

Write
$\sigma=(u,v),$
where
$u\colon 0\longrightarrow M,$ $v\colon S\longrightarrow S.$
Since $\pi\circ \sigma=\id_L$, necessarily
$v=\id_S.$
Moreover, any morphism of triples
$(u,v)\colon (X,Y,\phi)\longrightarrow (X',Y',\phi')$
must satisfy
$u\circ \phi=\phi'\circ (1_M\otimes v).$
Applying this to
$\sigma\colon (0,S,0)\longrightarrow (M,S,\mu),$
we obtain
$u\circ 0=\mu\circ (1_M\otimes \id_S),$
that is,
$0=\mu.$
This is impossible, since
$\mu\colon M\otimes_S S\longrightarrow M$
is an isomorphism. Therefore no such section $\sigma$ exists, and the above short exact
sequence does not split. Consequently, $L$ is not projective.

Hence
$\pd_T(L)=1.$

Since $L\in \FPR(T)$ and $\pd_T(L)=1$, we get
$\fPD(T)\ge 1.$
Combining this with the upper bound from Step~1, we conclude that
$\fPD(T)=1.$
\end{example}

\begin{corollary}\label{cor:UT2}
Let $R$ be a ring, and let
\[
UT_2(R)=\begin{pmatrix}R & R\\ 0 & R\end{pmatrix}.
\]
Then
$\fPD\bigl(UT_2(R)\bigr)\le \fPD(R)+1.$
\end{corollary}

\begin{proof}
Take $S=R$ and $M=R$, viewed as an $(R,R)$-bimodule in the natural way. Then
$M_R$ is free of rank one, hence projective, and ${}_RM$ is also free of rank one,
hence projective. Therefore Hypothesis~\ref{hyp:triangular} is satisfied. By
Theorem~\ref{thm:fPD-triangular}, we obtain
\[
\fPD\bigl(UT_2(R)\bigr)
\le
\max\{\fPD(R)+1,\fPD(R)+1\}
=
\fPD(R)+1.  \qedhere
\]
\end{proof}

\begin{corollary}\label{cor:diagonal-dominates}
Assume Hypothesis~\ref{hyp:triangular} and suppose that
$\fPD(R)\ge \fPD(S)+1.$
Then
$\fPD(R)\le \fPD(T)\le \fPD(R)+1.$
\end{corollary}

\begin{proof}
By Lemma~\ref{lem:pd-lower-R}, we have
$\fPD(T)\ge \fPD(R).$
On the other hand, by Theorem~\ref{thm:fPD-triangular},
\[
\fPD(T)\le \max\{\fPD(R)+1,\ \fPD(S)+1\}.
\]
Since
$\fPD(R)\ge \fPD(S)+1,$
it follows that
\[
\fPD(S)+1\le \fPD(R)\le \fPD(R)+1.
\]
Hence
$\fPD(T)\le \fPD(R)+1.$
Therefore
\[
\fPD(R)\le \fPD(T)\le \fPD(R)+1.  \qedhere
\]
\end{proof}

\begin{corollary}\label{cor:finiteness-ascends}
Assume Hypothesis~\ref{hyp:triangular}. If $\fPD(R)<\infty$ and $\fPD(S)<\infty$, then
\[
\fPD(T)<\infty.
\]
\end{corollary}

\begin{proof}
The conclusion follows immediately from Theorem~\ref{thm:fPD-triangular}, since
\[
\fPD(T)\le \max\{\fPD(R)+1,\ \fPD(S)+1\}.  \qedhere
\]
\end{proof}

\begin{corollary}\label{cor:UTn}
Let $R$ be a ring, and let $UT_n(R)$ be the ring of all $n\times n$ upper triangular
matrices over $R$. Then
\[
\fPD\bigl(UT_n(R)\bigr)\le \fPD(R)+(n-1).
\]
\end{corollary}

\begin{proof}
We argue by induction on $n$.

If $n=1$, then $UT_1(R)=R$, and the statement is trivial.

Now let $n\ge 2$, and set
$U:=UT_{n-1}(R).$
Then there is a ring isomorphism
\[
UT_n(R)\cong
\begin{pmatrix}
U & Ue\\
0 & R
\end{pmatrix},
\]
where $e=e_{\,n-1,n-1}\in U$ is the matrix unit in the \((n-1,n-1)\)-position.
Indeed, the upper-right corner consists of column vectors of length \(n-1\), which
identifies naturally with the \((U,R)\)-bimodule \(Ue\).

We verify Hypothesis~\ref{hyp:triangular} for the bimodule \(M:=Ue\).

First, as a left \(U\)-module, \(Ue\) is projective, since \(e\) is an idempotent and
\(Ue\) is a direct summand of the regular left \(U\)-module \(U\).

Second, as a right \(R\)-module, \(Ue\) is free of rank \(n-1\). Indeed, the last column
of a matrix in \(UT_{n-1}(R)\) can be chosen arbitrarily, so
$Ue\cong R^{\,n-1}$
as a right \(R\)-module. In particular, \(Ue_R\) is projective.

Therefore Hypothesis~\ref{hyp:triangular} is satisfied, and
Theorem~\ref{thm:fPD-triangular} yields
\[
\fPD\bigl(UT_n(R)\bigr)
\le
\max\{\fPD(U)+1,\ \fPD(R)+1\}.
\]
By the induction hypothesis,
\[
\fPD(U)=\fPD\bigl(UT_{n-1}(R)\bigr)\le \fPD(R)+(n-2).
\]
Hence
\[
\fPD\bigl(UT_n(R)\bigr)
\le
\max\{\fPD(R)+(n-2)+1,\ \fPD(R)+1\}
=
\fPD(R)+(n-1).
\]
This completes the induction.
\end{proof}

\noindent\textbf{Funding.}
This work was supported by the National Social Science Fund of China (No.~23EJY229),
the Open Fund of the Key Laboratory of Intelligent Optoelectronic Sensing Systems and
Applications in Sichuan Provincial Universities (No.~ZNGD2410), the 2023 Second-Batch
High-Level Talent Research Start-Up Program of Sichuan University of Arts and Science
(\emph{Finitisticization of Multiplicative Ideal Theory}, No.~2023RC009Z), and the
Mathematics and Finance Research Center Project of Sichuan University of Arts and Science
(No.~SCMF202516).  H. Kim was supported by Basic Science Research Program through the National Research Foundation of Korea(NRF) funded by the Ministry of Education (2021R1I1A3047469).

\medskip

\paragraph{{\bf Data availability:}}
 This manuscript has no associate data.

\medskip

\paragraph{{\bf Conflict of interest:}}
 The authors have no conflict of interest.

\end{document}